\documentclass[12pt]{iopart}
\usepackage{iopams}
\usepackage{amssymb}
\usepackage{geometry}
\usepackage[latin1]{inputenc}
\usepackage[italian,english]{babel}
\usepackage{
amsthm}
\usepackage{graphicx}
\usepackage{pgfplots}
\usepackage{paralist}
\eqnobysec
\usepackage{fancyhdr,hyperref}

\newtheorem{thm}{Theorem}[section]
\newtheorem{lem}[thm]{Lemma}

\newtheorem{prop}[thm]{Proposition}

\newtheorem{rem}[thm]{Remark}

\newcommand{\R}{\mathbb{R}}
\newcommand{\N}{\mathbb{N}}

\newcommand{\dive}{\rm{div}}
\newcommand{\curl}{\rm{curl}}
\newcommand{\hsol}{H_L(\Omega)}
\newcommand{\hsoln}{H_L^n(\Omega)}

\newenvironment{sistema}%
{\left\{\begin{array}{@{}l@{}}}{\end{array}\right.}
\patchcmd{\abstract}{\scshape\abstractname}{\textbf{\abstractname}}{}{}
\makeatletter 
\def\@makefnmark{} 
\makeatother 
 
\begin{document}
\title[The Monotonicity Principle for Magnetic Induction Tomography]{The Monotonicity Principle for Magnetic Induction Tomography}
\author{Antonello Tamburrino$^{1,2}$, Gianpaolo Piscitelli$^1$, Zhengfang Zhou$^3$}
\address{$^1$Dipartimento di Ingegneria Elettrica e dell'Informazione \lq\lq M. Scarano\rq\rq, Universit\`a degli Studi di Cassino e del Lazio Meridionale, Via G. Di Biasio n. 43, 03043 Cassino (FR), Italy.}
\address{$^2$Department of Electrical and Computer Engineering, Michigan State University, East Lansing, MI-48824, USA.}
\address{$^3$Department of Mathematics, Michigan State University, East Lansing, MI-48824, USA}
\eads{ \mailto{antonello.tamburrino@unicas.it} {\it(corresponding author)}, \mailto{gianpaolo.piscitelli@unicas.it}, \mailto{zfzhou@msu.edu}}
\vspace{10pt}

\begin{abstract}
The inverse problem dealt with in this article consists of reconstructing the electrical conductivity from the free response of the system in the magneto-quasi-stationary (MQS) limit. The MQS limit corresponds to a diffusion PDE. In this framework, a key role is played by the Monotonicity Principle (MP), that is a monotone relation connecting the unknown material property to the (measured) free-response. The MP is relevant as the basis of noniterative and real-time imaging methods.

The Monotonicity Principle has been found in many different physical problems governed by diverse PDEs. Despite its rather general nature, each physical/mathematical context requires the proper operator showing the MP to be identified. In order to achieve this, it is necessary to develop ad-hoc mathematical approaches tailored to the specific framework.

In this article, we prove that there exists a monotonic relationship between the electrical resistivity and the time constants characterizing the free response for MQS systems.
The key result is the representation of the induced current density through a modal representation. The main result is based on the analysis of an elliptic eigenvalue problem obtained from the separation of variables.
\end{abstract}

\section{Introduction}
In this paper, we study the inverse problem of reconstructing the electrical conductivity from the free response of Maxwell equations in the magneto-quasi-stationary (MQS) limit.
This limit corresponds to the case when (i) the displacement current appearing in the Amp\`{e}re-Maxwell law can be neglected, (ii) there is a conducting material and (iii) the magnetic energy of the system dominates the electrical energy \cite{haus1989electromagnetic}. In this limit a time-varying magnetic flux density is able to penetrate inside a conducting material inducing electrical currents known as eddy currents. In turn, the induced eddy currents produce a magnetic flux density which can be measured using sensors external to the conductor, thus enabling the tomographic capabilities (Eddy Current Tomography or Magneto Inductive Tomography) of material properties, such as the electrical resistivity (or, equivalently, conductivity) and magnetic permeability. Specifically, in this work we refer to the imaging of a conductor's electrical conductivity starting from the free response of the material, that is the response of the system when all sources have been switched off. We assume that the conducting material does not have magnetic properties. However, the extension to conductive and magnetic materials is straightforward.

Eddy Currents can be modeled in several ways. Here we choose to study the free response on an Eddy Current problem modeled by the following integral formulation (non-magnetic materials) \cite{albanese1997finite}:
\begin{eqnarray*}
\left\langle \eta \mathbf{J}, \mathbf{w}\right\rangle =-\partial_{t}\left\langle
A\mathbf{J}, \mathbf{w}\right\rangle \quad \forall\mathbf{w}\in\hsol,
\end{eqnarray*}
where $\bf J$ is the induced current density, $\Omega$ is the region occupied by the conducting material, $A$ is the compact, self-adjoint, positive definite operator
:
\begin{eqnarray}
\label{operat}
A:\mathbf{v} \in H_L(\Omega) \mapsto\frac{\mu_0}{4\pi}\int_\Omega\frac{\mathbf{v}(x^\prime) }{\vert\vert x-x^\prime\vert\vert}\ \textrm{d}V(x')\in L^2(\Omega;\R^3),
\end{eqnarray}
$\hsol$ is a proper functional space that will be defined later, $\eta$ is the electrical resistivity and $\mu_0$ is free space magnetic permeability.

In inverse problems, a key role is played by the Monotonicity Principle, based on a monotonic relationship connecting the unknown material property to the measured physical quantity. 
The Monotonicity Principle is fundamental for developing non-iterative imaging methods suitable for real-time imaging. It was introduced by A. Tamburrino and G. Rubinacci in \cite{Tamburrino_2002, Tamburrino2006FastMF,Tamburrino_2006}. The Monotonicity Principle has mainly been applied to inverse obstacle problems in which the aim is to reconstruct the shape of anomalies in a given background. In this setting, the method determines whether or not a test inclusion is part of an anomaly.
The three features rendering this test highly suitable are as follows: (i) the computational cost for processing a given test inclusion is negligible, (ii) the processing on different test inclusions can be carried out in parallel, and (iii) under proper assumptions, the MP provides rigorous upper and lower bounds to the unknown anomaly, even in the presence of noise (see \cite{Tamburrino2016284}, based on \cite{Tamburrino_2002,harrach2015resolution}).  {Moreover, it has been proved that the Monotonicity Principle provides a complete characterization of the unknown, under proper assumptions and in different settings (see \cite{harrach2013monotonicity,eberle2020shape,harrach2019monotonicity,griesmaier2018monotonicity,albicker2020monotonicity,daimon2020monotonicity, harrach2019monotonicity-based,harrach2020monotonicity-based,harrach2015combining}).}

The Monotonicity Principle appears to be a general feature which can be found in many different physical problems governed by diverse PDEs.
Although it was originally found for stationary PDEs arising from static problems (such as, Electrical Resistance Tomography, Electrical Capacitance Tomography, Eddy Current Tomography and Linear Elastostatics) \cite{Tamburrino_2002,Tamburrino_2006,harrach2013monotonicity,Calvano201232,Tamburrino2003233,eberle2020shape}, it was also found for stationary PDEs arising as the limit of quasi-static problems, such as Eddy Current Tomography for either large or small skin depth operations \cite{Tamburrino2006FastMF,Tamburrino_2006,Tamburrino_2010}.

The Monotonicity Principle has also been found and applied to tomography for problems governed by parabolic evolutive PDEs, see \cite{Su_2017,Tamburrino20161,Tamburrino2015159,Su2017,Tamburrino_2016_testing}. In this case the monotone operator is the one mapping the electrical resistivity into the (ordered) set of time constants for the natural modes.

Moreover, Monotonicity has been found in phenomena governed by Helmotz equations arising from wave propagation problems (hyperbolic evolutive PDEs) in time-harmonic operations (see \cite{AT_WAVE2015, harrach2019dimension, harrach2019monotonicity,griesmaier2018monotonicity,albicker2020monotonicity,daimon2020monotonicity}). Monotonicity for the Helmoltz equation arising from (steady-state) optical diffuse tomography was introduced in \cite{meftahi2020uniqueness}. 

Monotonicity has also been proved for nonlinear generalizations of elliptic PDEs arising from static phenomena. The most comprehensive contribution is \cite{corboesposito2020monotonicity}, which deals with the nonlinear elliptic case. 
The application setting refers to Electrical Resistance Tomography, where the electrical resistivity is nonlinear. The specific case of $p$-Laplacian is also treated in \cite{brander2018monotonicity} and \cite{guo2016inverse}.


Furthermore, Monotonicity is proved for the nonlocal fractional Schr\"odinger equation in \cite{harrach2019monotonicity-based,  harrach2020monotonicity-based}.

The concept of regularization for the Monotonicity Principle Method (MPM) is introduced in \cite{Rubinacci20061179,garde2017convergence}. It is worth noting that imaging via the Monotonicity Principle cannot be based upon the minimization of an objective function, where regularization can be easily introduced by means of penalty terms. Vice versa, the Monotonicity Principle is also used  as a regularizer in \cite{harrach2016enhancing}.

The MPM has been applied to many different engineering problems. A first experimental validation of the MPM for Eddy Current Tomography is shown in  \cite{Tamburrino201226}. Monotonicity is combined with frequency-difference
and ultrasound-modulated Electrical Impedance Tomography measurements in \cite{harrach2015combining}.
The authors in \cite{flores2010electrical} use a priori information obtained from Breast Microwave Radar images to estimate the location of dense breast regions.
Harrach and Minh in \cite{harrach2018monotonicity} give a new algorithm to improve the quality of the reconstructed images in electrical impedance tomography together with numerical results for experiments on a standard phantom.
\cite{soleimani2007monotonicity} proposes an application of the MPM for two-phase materials in Electrical Capacitance Tomography, Electrical Resistance Tomography and Magneto Inductance Tomography. 
Other results of the MPM applied to Tomography and Nondestructive Testing, can be found in \cite{ventre2016design,soleimani2006shape}. Furthermore, \cite{zhou2018monotonicity} presents a monotonicity-based spatiotemporal conductivity imaging method for continuous regional lung monitoring using electrical impedance tomography (EIT). The MPM has also been applied to the homogenization of materials \cite{7559815} and the inspection of concrete rebars \cite{DeMagistris2007389,Rubinacci2007333}.

The imaging methods based on the Monotonicity Principle fall within the class of non-iterative imaging methods. Colton and Kirsch introduced the first non-iterative approach named Linear Sampling Method (LSM) \cite{Colton_1996}, then Kirsch proposed the Factorization Method (FM)  \cite{Kirsch_1998}. Ikeata proposed the Enclosure Method \cite{ikehata1999draw,Ikehata_2000} and Devaney applied MUSIC (MUltiple SIgnal Classification), a well-known algorithm in signal processing, as an imaging method \cite{Devaney2000}. In \cite{fernandez2019noniterative}, a non-iterative method based on the concept of topological derivatives is proposed to find the shape of anomalies in an otherwise homogeneous material.

In this work, we prove that the electrical current density in the absence of the source (source-free response) can be represented through a modal decomposition:
\begin{eqnarray}
\label{eq20Su}
\mathbf{J}\left(x,t\right)  =\sum_{n=1}^\infty c_n\ {\bf j}_n (x)e^{-t/\tau_n(\eta)}\quad\textrm{in}\ \Omega\times [0,+\infty[.
\end{eqnarray}
In modal decomposition (\ref{eq20Su}), $\mathbf{j}_{n}(x)$ is a mode and $\tau_n(\eta)>0$ is the corresponding time constant, $\forall n \in \N$. Each mode and its related time constant depend on the electrical resistivity $\eta$. In Proposition \ref{prop_complete} we prove that the sequence of modes $\left\{  \mathbf{j}_{n}\right\}  _{n\in\mathbb{N}}$ is a complete basis.
Moreover, we prove that $\tau_n(\eta)\to 0$ as $n\to\infty$ (Proposition \ref{infinitesima}). Equation (\ref{eq20Su}) allows us to generalize the representation in \cite[eq. (20)]{Su_2017}, valid for the discrete case, to the continuous case. 

Moreover, in Theorem \ref{monotonicity_thm}, we prove the Monotonicity Principle for the sequence of time constants $\{\tau_n(\eta)\}_{n\in\N}$:
\begin{eqnarray}
\label{monotonicity_tc}
\eta_1 \leq\eta_{2}\quad \Rightarrow\quad\tau_{n}\left(  \eta_{1}\right)  \geq\tau
_{n}\left(  \eta_{2}\right)  \quad \forall n\in
\mathbb{N},
\end{eqnarray}
where $\eta_1\leq\eta_2$ means $\eta_1(x)\leq \eta_2(x)$ for a.e. $x\in\Omega$. The time constants $\tau_n(\eta)$ appearing in (\ref{monotonicity_tc}) must be ordered monotonically. Hereafter, we assume they are placed in decreasing order.

The original content of this paper consists in extending the Monotonicity Principle for time constants from the discrete setting of \cite{Su_2017} to the continuous case. This extension requires methods and techniques which are completely different from those previously used for the discrete setting. Moreover, this paper provides the mathematical foundations to justify the truncated version of (\ref{eq20Su}), which underpins the discrete setting. Specifically, the convergence to zero of the time constants allows series expansion (\ref{eq20Su}) to be truncated to a proper finite number of terms.

The paper is organized as follows. In Section \ref{FoP}, we describe the target inverse problem together with the mathematical model of the underlying physics. In Section \ref{TEP} we study the eigenvalue problem which gives rise to time constants and, in Section \ref{MoE}, we prove the main result: the Monotonicity Principle for time constants. In Section \ref{Dfr} we provide a discussion of the results and, finally, in Section \ref{C} we draw some conclusions.

\section{Statement of the Problem}\label{FoP}
The reference problem in Magnetic Induction Tomography (MIT) consists in retrieving the spatial distribution of the electrical resistivity of a material by means of electromagnetic fields.

The electromagnetic field is generated by time-varying electrical currents circulating in a proper set of coils (see Figure \ref{fig_schema}). These time-varying currents produce a time-varying magnetic flux density $\textbf{B}(x,t)$ that induces an electrical field $\textbf{E}(x,t)$ and, consequently, an electrical current density $\textbf{J}(x,t)$ in the conducting domain $\Omega$ \cite{haus1989electromagnetic}. The electrical resistivity $\eta$ affects the induced current density $\textbf{J}(x,t)$ which, in turns, produces a \lq\lq reaction\rq\rq magnetic flux density $\textbf{B}^{eddy}(x,t)$. In MIT the measurement of $\textbf{B}^{eddy}(x,t)$ carried out externally to the conducting domain, makes it possible, in principle, to reconstruct the unknown $\eta$.

\begin{figure}[!ht]	\centering	\includegraphics[width=0.7\textwidth]{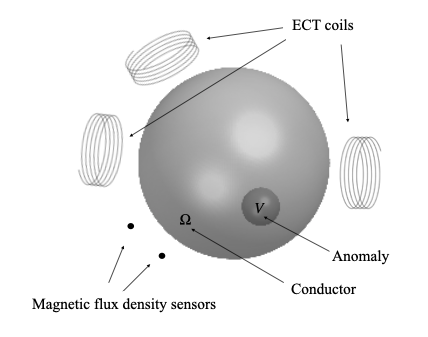}	\caption{Representation of a typical Magnetic induction tomography.}	\label{fig_schema}\end{figure}

We have two types of measurements related to $\textbf{B}^{eddy}$. The first consists in measuring $\textbf{B}^{eddy}$ with a magnetic flux density sensor, the second consists in measuring the induced voltage $v_{eddy}(t)$ across a pick-up coil. It is worth noting that $v^{eddy}=d\varphi^{eddy}/{dt}$ where $\varphi^{eddy}$ is the magnetic flux linked with the pick-up coil. Moreover, these quantities share the same set of time constants of $\mathbf J$. Indeed, by applying the Biot-Savart law to the source free response (\ref{eq20Su}), we have
\begin{eqnarray}\label{eq21Su}
\mathbf{B}^{eddy}\left(x,t\right)=\sum_{n=1}^\infty c_n {\bf b}_n (x)e^{-t/\tau_n(\eta)}\quad\textrm{in}\ \Omega\times [0,+\infty[
\end{eqnarray}
and
\begin{eqnarray}\label{eq22Su}
v^{eddy}(t)=\frac{d\varphi^{eddy}(t)}{dt}=\sum_{n=1}^\infty c_n v_{n}e^{-t/\tau_n(\eta)}\quad\textrm{in}\ \Omega\times [0,+\infty[.
\end{eqnarray}
Summing up, the protocol entails in gathering the waveform of either $\textbf{B}^{eddy}$ or $v^{eddy}$. Then, the waveforms are pre-processed to extract the time constants and, finally, the set of time constants is provided as input to the imaging algorithm.

\subsection{Mathematical Model for the Eddy Current Problem}
In this Section, for the sake of completeness, we summarize the mathematical model for the Eddy Current problem.

Throughout this paper, $\Omega$ is the region occupied by the conducting material. We assume $\Omega\subset\R^3$ to be an open bounded domain with a Lipschitz boundary and outer unit normal $\hat {\bf n}$.
We denote by $V$ and $S$ the $3$-dimensional and the bi-dimensional Hausdorff measure in $\R^3$, respectively and by $\langle\cdot,\cdot\rangle$ the usual $L^2$-integral product on $\Omega$. 

 {Hereafter we refer to the following functional spaces
\begin{eqnarray*}
L^\infty_+(\Omega):=\{\theta\in L^\infty(\Omega) \ | \ \theta\geq c_0\ \mathrm{a.e.\ in} \ \Omega \ \mathrm{for}\ c_0>0\},\\
H_{\mathrm{div}}(\Omega):=\{ {\bf v} \in L^2(\Omega;\R^3) \ |\ \mathrm{div} ({\bf v}) \in L^2(\Omega)\},\\
H_{\mathrm{curl}}(\Omega):=\{ {\bf v} \in L^2(\Omega;\R^3) \ |\ \mathrm{curl} ({\bf v}) \in L^2(\Omega)\},\\
\hsol  :=\left\{  \mathbf{v}\in H_{\rm div}(\Omega)\   | \ \dive (\mathbf{v})=0\ \mathrm{in}\ \Omega,\ \mathbf{v}\cdot\hat{\bf n}=0\ \mathrm{on}\ \partial \Omega\right\},
\end{eqnarray*}
and to the derived spaces $L^2(0,T;H_L(\Omega))$ and $L^2(0,T;H_{\rm curl}(\Omega))$, for any $0<T< +\infty$.
}

Let ${\bf E}$, ${\bf B}$, ${\bf H}$ and ${\bf J}$ be the electric field, the magnetic flux density, the magnetic field and the electrical current density, respectively. The Magneto-Quasi-Stationary approximation of Maxwell's equations is described by (see, for instance, \cite{touzani2014mathematical}):
\begin{eqnarray}
\label{third_max}
{\curl}\ {\mathbf{E}}(x,t) =-\partial_t{\bf B}(x,t)\  &\textrm{in}&\ \R^3\times [0,+\infty[,
\\
\label{fourth_max}
{\curl}\ {\bf H}(x,t) = {\bf J} (x,t)+{\bf J}_s(x,t)
\ &\textrm{in}& \ \R^3\times [0,+\infty[ ,
\\
\label{second_max}
{\dive}\ {\bf B}(x,t)  = 0\ &\textrm{in}&\ \R^3\times [0,+\infty[,
\end{eqnarray}
\vspace{-0.7cm}
\begin{eqnarray}
{\bf J} (x,t) = 
\begin{sistema}\label{Ohm} 
 \eta^{-1}(x) \ {\bf E}(x,t)\qquad\ \ \textrm{in}\ \Omega\times [0,+\infty[\\
 0 \qquad\qquad\qquad\qquad \textrm{in} \ (\R^3\setminus\Omega)\times [0,+\infty[,
\end{sistema}
\end{eqnarray}
\vspace{-0.5cm}
\begin{eqnarray}
\label{dpermH}
{\bf B}(x,t)=\mu_0 {\bf H}(x,t)\qquad\qquad\quad\ \textrm{in}\ \Omega\times [0,+\infty[,
\end{eqnarray}
\begin{eqnarray}\label{inf_van}
|{\bf B}(x,t)|=O(|x|^{-3})\ \textrm{uniformly for $t\in[0,+\infty[$, as} \ |x| \to +\infty,
\end{eqnarray}
where  $\eta\in L_+^\infty(\Omega)$ is the electrical resistivity of the conductor, $\mu_0$ is the magnetic permeability of the free space, ${\bf J}_s$ is the prescribed source current density and we have assumed that there are no magnetic materials.
Equations (\ref{third_max}), (\ref{fourth_max}) and (\ref{second_max}) are to be understood in the weak form in order to implicitly take into account the jumps of the fields at material interfaces \cite{bossavit1998computational}.
Problem (\ref{third_max})-(\ref{inf_van}) is equivalent to an integral formulation in the weak form (see \cite{albanese1997finite,bossavit1981numerical}):
\begin{eqnarray}
\label{weak_form}
\left\langle \eta \mathbf{J}, \mathbf{w}\right\rangle =-\langle\partial_{t}
A\mathbf{J},\mathbf{w}\rangle-\langle\partial_{t}{\bf A}_S, \mathbf{w}\rangle \quad \forall\mathbf{w}\in H_L(\Omega),
\end{eqnarray}
where ${\bf A}_S\in L^2(0,T;H_{\rm curl}(\Omega))$ is the vector potential produced by the prescribed source current ${\bf J}_s\in L^2(0,T;H_{L}(\Omega_s))$, $\Omega_s$ is a bounded open set with Lipschitz boundary,
$A$ is the operator
defined in (\ref{operat}) 
and ${\bf J}\in L^2(0,T;\hsol)$.

We refer to \cite{bossavit1998computational,bossavit1981numerical} for $\hsol$ and a general discussion on functional spaces related to Maxwell equations.
The functional space $\hsol$ was applied in the development of numerical methods for the solution of integral source formulations of Maxwell equations in \cite{albanese1997finite, forestiere2017frequency}. The elements of $\hsol$ are vector fields with closed streamlines (see also Figure \ref{fig4}).
\begin{figure}[!ht]
	\centering
	\includegraphics[width=0.45\textwidth]{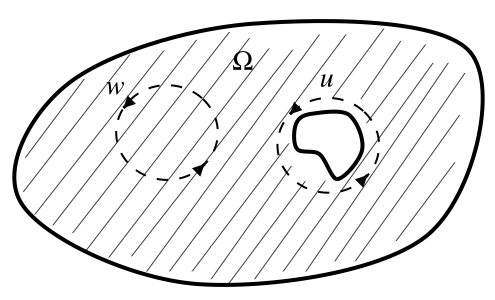}
		\caption{The domain $\Omega$ (dashed region) and two typical elements $\bf w$ and $\bf u$ of $\hsol$. Both vector fields present closed streamlines. The streamlines for $\bf w$ are homotopic to a point, whereas the streamlines for $\bf u$ circulate around the cavity.}
	\label{fig4}
\end{figure}

\section{The Eigenvalue Problem}\label{TEP}

This Section provides the characterization of the time constants. Specifically, we prove that they are the eigenvalues for a proper generalized eigenvalue problem. Subsection \ref{TFE} gives the variational characterization of the first eigenvalue, whereas Subsection \ref{THE} deals with higher eigenvalues.

To find this generalized eigenvalue problem, we notice that in the absence of source currents (\ref{weak_form}) reduces to:
\begin{eqnarray*}
\left\langle \eta \mathbf{J}, \mathbf{w}\right\rangle =-\partial_{t}\langle
A\mathbf{J},\mathbf{w}\rangle\quad\forall\mathbf{w}\in \hsol).
\end{eqnarray*}
Then, the separation of variables $\mathbf{J}\left(x,t\right)  =i\left(  t\right)  \mathbf{j}\left(  x\right)$ gives
\begin{equation*} 
i\left(  t\right)  \left\langle \eta(x)\mathbf{j}(x),\mathbf{w}(x)\right\rangle
=-i'(t) \left\langle A\mathbf{j}(x),\mathbf{w}(x)\right\rangle \quad \forall\mathbf{w}\in \hsol,
\end{equation*}
and, therefore,%
\begin{eqnarray*}
\frac{i\left(  t\right)}{i'\left(  t\right)} =-\tau(\eta)\quad\forall t\in]0,+\infty[,\\
\frac{\left\langle A \mathbf{j},\mathbf{w}\right\rangle
}{\left\langle\eta\mathbf{j}, \mathbf{w}\right\rangle }  =\tau(\eta)
\quad \forall\mathbf{w}\in\hsol,
\end{eqnarray*}
where $\tau(\eta)$ is the separation constant. Therefore, $\mathbf{J}\left(x,t\right)  =e^{-t/\tau(\eta)}\mathbf{j}\left(  x\right)  $ and%
\begin{equation}
\left\langle A\mathbf{j},\mathbf{w}\right\rangle =\tau(\eta)\left\langle \eta\mathbf{j}, \mathbf{w}\right\rangle\quad \forall\ \mathbf{w}\in\hsol .\label{eq2}%
\end{equation}
When (\ref{eq2}) admits a nonzero solution, the real number $\tau$ and the function $\mathbf{j}\in \hsol\setminus\{\mathbf 0\}$ are called the {\it eigenvalue} and {\it eigenfunction} of (\ref{eq2}), respectively.


In this Section,  we prove that:
\begin{enumerate}
\item the generalized eigenvalues and eigenvectors form countable sets:
$\left\{  \tau_{n}(\eta)\right\}  _{n\in%
\mathbb{N}
}$ and $\left\{  \mathbf{j}_{n}\right\}  _{n\in%
\mathbb{N}
}$;
\item the eigenvalues can be ordered such that $\tau_{n}(\eta)\geq\tau_{n+1}(\eta)$;
\item $\tau_{n}(\eta)>0$ and $\lim_{n\rightarrow+\infty}\tau_{n}(\eta)=0$;
\item the set of eigenvectors $\{\mathbf{j}_{n}\}_{n\in\N}$ forms a complete basis in $\hsol$.
\end{enumerate}



In order to prove these properties, the key assumptions are that $A$ is a positive definite, compact and self-adjoint operator and $\eta$ is a positive definite and bounded operator (see \cite{nirenberg1961lecture,armitage2012classical}).

The approach presented hereafter is inspired by the Rayleigh-Ritz minimum principle and the related Min-Max Principle (see Courant and Hilbert \cite[Chap. V]{courant1966methods}, Davies \cite[Chap. I]{davies1996spectral}, Kesavan  \cite[Chap. 4]{kesavan1989topics}, Reed and Simon \cite[Chap. XIII]{reed1978iv}).
Specifically, the Theorem 3.6.2 in \cite{kesavan1989topics} provides the three characterizations for linear eigenvalues that we use throughout this paper. They are based on the concepts of span, orthogonality and fixed dimension.

Throughout this paper, $\left\Vert\mathbf{v}\right\Vert_{\eta}$ denotes the weighted norm
\begin{eqnarray*}
\left\Vert\mathbf{v}\right\Vert_{\eta}:=\sqrt{\left\langle\eta \mathbf{v},\mathbf{v}\right\rangle }, \quad \forall \mathbf{v} \in{L}^{2}\left(\Omega;\R^3\right).
\end{eqnarray*}

For the sake of simplicity, hereafter we omit the dependence of $\bf j$ on $\eta$, and we omit the dependence of $\tau$ on $\eta$ in the proofs.

\subsection{The First Eigenvalue}\label{TFE}
In order to compute the first eigenvalue of (\ref{eq2}), we use the variational characterization:
\begin{eqnarray}
\label{minimum}
\tau_{1}(\eta)=\max_{\left\Vert \mathbf{j}\right\Vert _{\eta}=1, \ \mathbf{j}\in\hsol}\langle A  \mathbf{j},\mathbf{j}\rangle.
\end{eqnarray}

In the following we exploit the fact that $\hsol$ is weakly closed in ${L}^{2}\left(\Omega;\R^3\right)$. 
Indeed, we have the following Lemma. 
\begin{lem}\label{lemma01}
Let $\mathbf{v}_{j}\in\hsol$. If $\mathbf{v}_{j}\rightharpoonup\mathbf{v}$ in ${L}^{2}\left(  \Omega;\R^3\right)$, then $\mathbf{v}\in \hsol$.
\end{lem}
\begin{proof}
If $\varphi\in H^{1}(\Omega)$, then $\mathbf{w}=\nabla
\varphi\in{L}^{2}\left(  \Omega;\R^3\right)$ and, by a Divergence Theorem, we have 
\begin{eqnarray*}
\int_{\Omega}\mathbf{v}_{j}\cdot\nabla\varphi\ \mathrm{d}V=\int_{\partial \Omega}%
\varphi\left(  \mathbf{v}_{j}\cdot\mathbf{\hat{n}}\right)\  \mathrm{d}S-\int%
_{\Omega}\varphi\ \dive\left(\mathbf{v}_{j}\right) \mathrm{d}{\it V}=0.
\end{eqnarray*}
Thus, 
it results that%
\[
\int_{\Omega}\mathbf{v}\cdot\nabla\varphi\ \mathrm{d}V=0\quad \forall\varphi\in H^{1}\left(\Omega\right).
\]
Therefore, $\dive (\mathbf{v})=0$ in $\Omega$ and $\mathbf{v}\cdot\mathbf{\hat{n}}=0$ on $\partial \Omega$, i.e. $\mathbf{v}\in\hsol$.
\end{proof}

The first eigenvalue and eigenvector are characterized by the following Proposition.

\begin{prop}
\label{lemma02}
Let $\Omega$ be an open bounded domain in $\R^3$, with Lipschitz boundary. Then problem (\ref{minimum}) admits a maximum, i.e. there exists $\mathbf{j}_{1}\in L^2\left(\Omega;\R^3\right)$ such
that\newline(i) $\mathbf{j}_{1}\mathbf{\in}\hsol $%
;\newline(ii) $\langle A  \mathbf{j}_{1},\mathbf{j}_{1}\rangle  =\tau_{1}(\eta)$%
;\newline(iii) $\left\Vert \mathbf{j}_{1}\right\Vert _{\eta}=1$;
\newline(iv) $\tau_1$ is an eigenvalue of (\ref{eq2}) and $\mathbf{j}_1$ is the associated eigenvector, i.e.
\[
\left\langle A\mathbf{j}_{1},\mathbf{w}\right\rangle =\tau
_{1}(\eta)\left\langle \eta\mathbf{j}_{1},\mathbf{w}\right\rangle \quad
\ \forall\mathbf{w}\in \hsol.
\]
\end{prop}
\begin{proof}
By definition of supremum, $\exists\left\{  \mathbf{v}_{j}\right\}_{j\in\mathbb{N}}$ in $\hsol  \subset{L}^{2}\left(  \Omega;\R^3\right)$ such that $\left\Vert \mathbf{v}_{j}\right\Vert _{\eta}=1$ and $\langle A\mathbf{v}_{j},\mathbf{v}_{j}\rangle  \rightarrow\tau_{1}$. 
This sequence is bounded in ${L}^{2}\left(\Omega;\R^3\right)$.
Indeed, if we set $\eta_L:=\mathrm{ess\ inf}_{\Omega}\eta$, then $\left\Vert \mathbf{v}_{j}\right\Vert \leq\left(\eta_L\right)^{-1/2}\left\Vert \mathbf{v}_{j}\right\Vert _{\eta}=\left( \eta_L\right)^{-1/2}$. 
Since $\left\{  \mathbf{v}_{j}\right\}  _{j\in\mathbb{N}}$ is a bounded sequence and ${A}$ is a compact operator, there exists a subsequence $\left\{  \mathbf{v}_{j_{k}}\right\}  _{k\in
\mathbb{N}}$ such that ${A}\mathbf{v}_{j_{k}}\rightarrow\mathbf{z}\in
{L}^{2}\left(\Omega;\R^n\right)$.
Since any bounded sequence in
${L}^{2}\left(\Omega;\R^3\right)  $ admits a weakly converging subsequence, we can extract from $\left\{\mathbf{v}_{j_{k}}\right\}  _{k\in\mathbb{N}}$ another subsequence $\left\{  \mathbf{w}_{l}\right\}_{l\in\mathbb{N}} =\{  \mathbf{v}_{j_{k_l}}\}  _{l\in\mathbb{N}}$ such that $\mathbf{w}_{l}$ $\rightharpoonup\mathbf{j}_{1}$ in
${L}^{2}\left(  \Omega;\R^3\right)$. Summing up, we have proved:
\begin{eqnarray*}\eqalign{
\left\Vert \mathbf{w}_{l}\right\Vert _{\eta}  & =1;\\
\lim_{l\rightarrow+\infty}\langle A  \mathbf{w}_{l},\mathbf{w}_{l}\rangle   &=\tau_{1};\\
\lim_{l\rightarrow+\infty}A\mathbf{w}_{l}  & =\mathbf{z}\in{L}^{2}\left(\Omega;\R^3\right); \\
\lim_{l\rightarrow+\infty}\left\langle \mathbf{w}_{l},\mathbf{w}\right\rangle & =\left\langle \mathbf{j}_{1},\mathbf{w}\right\rangle ,\quad \forall\mathbf{w\in}{L}^{2}\left(  \Omega;\R^3\right).}
\end{eqnarray*}
The claim {\it (i)} follows from the fact that $\mathbf{w}_{l}$ $\rightharpoonup\mathbf{j}_{1}$ in ${L}^{2}\left(\Omega;\R^3\right)  $ implies $\mathbf{j}_{1}\mathbf{\in}%
\hsol$, thanks to Lemma \ref{lemma01}.
\newline In order to prove {\it (ii)}, we observe that $\left\langle
A\mathbf{w}_{l},\mathbf{w}_{l}\right\rangle =\left\langle
A\mathbf{w}_{l}-\mathbf{z},\mathbf{w}_{l}\right\rangle
+\left\langle \mathbf{z},\mathbf{w}_{l}\right\rangle$. Thus
\[
\tau_{1}=\lim_{l\rightarrow+\infty}\langle A \mathbf{w}_l,  \mathbf{w}_{l}\rangle
=\lim_{l\rightarrow+\infty}\left\langle \mathbf{z},\mathbf{w}_{l}\right\rangle
=\left\langle \mathbf{z},\mathbf{j}_{1}\right\rangle .
\]
Furthermore, since $A$ is a self-adjoint operator and $A\mathbf w_l\to\mathbf{z}$, we have
\[\fl
\tau_1=\left\langle \mathbf{z},\mathbf{j}_{1}\right\rangle =\lim_{l\rightarrow
+\infty}\left\langle A\mathbf{w}_{l},\mathbf{j}_{1}\right\rangle
=\lim_{l\rightarrow+\infty}\left\langle \mathbf{w}_{l},A\mathbf{j}_{1}\right\rangle =\left\langle \mathbf{j}_{1},A\mathbf{j}_{1}\right\rangle.
\]
Finally, we can prove {\it (iii)}. 
We first observe that%
\begin{equation*}
\left\Vert \mathbf{j}_{1}\right\Vert _{\eta}^{2}=\lim_{l\rightarrow+\infty
}\left\langle \mathbf{w}_{l},\eta\mathbf{j}_{1}\right\rangle 
\end{equation*}
and that $\left\vert \left\langle \mathbf{w}_{l},\eta\mathbf{j}_{1}%
\right\rangle \right\vert \leq\left\Vert \sqrt{\eta}\mathbf{w}_{l}\right\Vert
\left\Vert \sqrt{\eta}\mathbf{j}_{1}\right\Vert =\left\Vert \mathbf{j}%
_{1}\right\Vert _{\eta}$. 
Thus $\left\langle \mathbf{w}_{l}%
,\eta\mathbf{j}_{1}\right\rangle \leq\left\Vert \mathbf{j}_{1}\right\Vert
_{\eta}$ and%
\[
\left\Vert \mathbf{j}_{1}\right\Vert _{\eta}^{2}=\lim_{l\rightarrow+\infty
}\left\langle \mathbf{w}_{l},\eta\mathbf{j}_{1}\right\rangle \leq\left\Vert
\mathbf{j}_{1}\right\Vert _{\eta}
\]
that is $\left\Vert \mathbf{j}_{1}\right\Vert _{\eta}\leq1$. The reverse inequality follows by observing that
\[
\tau_{1}\geq\left\langle A  \frac{\mathbf{j}_{1}}{\left\Vert \mathbf{j}%
_{1}\right\Vert _{\eta}},\frac{\mathbf{j}_{1}}{\left\Vert \mathbf{j}%
_{1}\right\Vert _{\eta}}\right\rangle  =\frac{1}{\left\Vert \mathbf{j}%
_{1}\right\Vert _{\eta}^2}\langle A  \mathbf{j}_{1}, \mathbf{j}_{1}\rangle  =\frac
{\tau_{1}}{\left\Vert \mathbf{j}_{1}\right\Vert _{\eta}^2}.
\]
Now {\it (iv)} needs to be proved. Let $\mathbf{w\in}\hsol$, we set
\[
h_{\mathbf{w}}\left(  t\right)  =\left\langle A \frac{\mathbf{j}%
_{1}+t\mathbf{w}}{\left\Vert \mathbf{j}_{1}+t\mathbf{w}\right\Vert _{\eta}},  \frac{\mathbf{j}%
_{1}+t\mathbf{w}}{\left\Vert \mathbf{j}_{1}+t\mathbf{w}\right\Vert _{\eta}%
}\right\rangle  .
\]
First, we notice that $\frac{\mathbf{j}_{1}+t\mathbf{w}}{\left\Vert
\mathbf{j}_{1}+t\mathbf{w}\right\Vert _{\eta}}\mathbf{\in}\hsol$ and that the $\left\Vert \mathbf{\cdot}\right\Vert _{\eta}$-norm is
unitary for any real $t$. Thus $h_{\mathbf{w}}\left(  t\right)  \leq\tau_{1}(\Omega)$
and achieves its maximum at $t=0$. As a consequence, $h_{\mathbf{w}}^{\prime
}\left(  0\right)  =0$, which gives%
\[
\left\langle {A}\mathbf{j}_{1},\mathbf{w}\right\rangle -\left\langle {\eta}%
\mathbf{j}_{1},\mathbf{w}\right\rangle \left\langle A
\mathbf{j}_{1},\mathbf{j}_{1}\right\rangle =0,
\]
where the self-adjointness of operator $A$ is used.
Hence, we conclude
\[
\left\langle A\mathbf{j}_{1},\mathbf{w}\right\rangle =\tau_{1}\left\langle\mathcal{\eta}\mathbf{j}_{1},\mathbf{w}\right\rangle\quad \forall\mathbf{w\in}\hsol.
\]
\end{proof}
We remark that these properties have been investigated also for nonlinear and nonlocal operators (see e.g. \cite{piscitelli2016anonlocal} and reference therein).

\subsection{The Higher Eigenvalues}
\label{THE}
Here we characterize higher eigenvalues (and eigenfunctions) for problem (\ref{eq2}) by the means of an iterative process. Various characterizations are possible for higher eigenvalues (see \cite{della2019second,della2017sharp} and reference therein). By induction, we define
\begin{eqnarray}
\label{n-minprob}
\tau_{n}(\eta):=\max_{\left\Vert \mathbf{j}\right\Vert _{\eta}=1,\ \mathbf{j}\in\hsoln}\langle A  \mathbf{j},  \mathbf{j}\rangle\quad n=1,2,...  
\end{eqnarray}
where
\begin{eqnarray*}
 H_L^1\left(  \Omega\right)   & :=\hsol \\
  \hsoln & :=\left\{  \mathbf{v}\in \hsol  |\left\langle \mathbf{v,}\eta\mathbf{j}_{1}\right\rangle=\ldots=\left\langle \mathbf{v,}\eta\mathbf{j}_{n-1}\right\rangle =0\right\}\ n=2,3,...
\end{eqnarray*}
The following properties for (\ref{n-minprob}) hold.
\begin{prop}
\label{lemma04} 
Let $\Omega$ be an open bounded domain with Lipschitz boundary. Then problem (\ref{n-minprob}) admits a maximum for any $n\in\mathbb{N}$, i.e. there exists $\mathbf{j}_{n}{\in}{L}^{2}\left(  \Omega;\R^3\right)  $ such
that\newline(i) $\mathbf{j}_{n}\mathbf{\in} \hsoln
$;\newline(ii) $\langle A  \mathbf{j}_{n},  \mathbf{j}_{n}\rangle  =\tau_{n}(\eta)$%
;\newline(iii) $\left\Vert \mathbf{j}_{n}\right\Vert _{\eta}=1$;
\newline(iv) $\tau_n$ is an eigenvalue of (\ref{eq2}) for $n=2,3,...$ and $\mathbf{j}_n$ is the associated eigenvector:
\begin{equation}
    \label{weak_higher}
\left\langle A\mathbf{j}_{n}, \mathbf{w}\right\rangle =\tau
_{n}(\eta)\left\langle \eta\mathbf{j}_{n},\mathbf{w}\right\rangle \quad
\ \forall\mathbf{w}\in \hsol;
\end{equation}
(v) The elements ${\bf j}_n$ are orthogonal with respect to $\eta$, i.e. $\langle \mathbf{j}_n,\eta\mathbf{j}_{m}\rangle=0\quad \forall\ n\neq m$;
\newline(vi) The elements ${\bf j}_n$ are orthogonal with respect to $A$, i.e. $\langle \mathbf{j}_n,A\mathbf{j}_m\rangle=0\quad \forall\ n\neq m$.
\end{prop}
\begin{proof}
Similarly to Proposition \ref{lemma02}, it can be shown that for any $n\in\mathbb N$ there exists a sequence $\left\{\mathbf{v}_{l}\right\}  _{l\in\mathbb{N}}$ in $\hsoln$ such that
\begin{eqnarray*}
\left\Vert \mathbf{v}_{l}\right\Vert _{\eta}   =1;\\
\lim_{l\rightarrow+\infty}\langle A  \mathbf{v}_{l},  \mathbf{v}_{l}\rangle   =\tau_{n};\\
\lim_{l\rightarrow+\infty}A\mathbf{v}_{l} =\mathbf{z}\in{L}^{2}\left(\Omega;\R^3\right) ;\\
\lim_{l\rightarrow+\infty}\left\langle \mathbf{v}_{l},\mathbf{w}\right\rangle =\left\langle \mathbf{j}_{n},\mathbf{w}\right\rangle \quad \forall\mathbf{w\in}{L}^{2}\left(\Omega;\R^3\right).
\end{eqnarray*}
To prove {\it (i)}, we notice that $\mathbf{v}_{l}$ $\rightharpoonup\mathbf{j}_{n} $
in ${L}^{2}\left(  \Omega;\R^3\right)  $, thus $\mathbf{j}_{n}\mathbf{\in}%
\hsol  $ thanks to Lemma \ref{lemma01}. Moreover,
$\left\langle \mathbf{v}_{l}\mathbf{,}\eta\mathbf{j}_{m}\right\rangle =0$ for
$m=1,\ldots,n-1$. Thus,
\[
\left\langle \mathbf{j}_{n},\eta\mathbf{j}_{m}\right\rangle =\lim
_{l\rightarrow+\infty}\left\langle \mathbf{v}_{l},\eta\mathbf{j}%
_{m}\right\rangle =0,\ m=1,\ldots,n-1
\]
that is $\mathbf{j}_{n}\mathbf{\in}\hsol   $.\newline The
proofs of {\it (ii)} and {\it (iii)} follow those provided in Proposition \ref{lemma02}%
.\newline The proof of {\it (iv)} follows the approach of Proposition \ref{lemma02}{\it (iv)}. However,
with that approach it can be proved that $\left\langle A\mathbf{j}_{n},\mathbf{w}\right\rangle =\tau_{n}\left\langle \eta\mathbf{j}_{n},\mathbf{w}\right\rangle\ \forall\mathbf{w}\in\hsoln$. In order to prove that {\it (iv)} holds in $\hsol$, we notice that $\hsol =\hsoln
\oplus span\left\{  \mathbf{j}_{1},\ldots,\mathbf{j}_{n-1}\right\}  $, that is
any $\mathbf{v}\mathbf{\in}\hsol$ can be written as
\[
\mathbf{v=v}^{n}+\sum_{m=1}^{n-1}\lambda_{m}\mathbf{j}_{m}\quad \textrm{where}\quad \mathbf{v}
^{n}\in \hsoln,\ \lambda_{m}\in\mathbb{R}.
\]
Therefore, it follows that%
\begin{eqnarray*}
\left\langle \mathbf{v},A\mathbf{j}_{n}\right\rangle  &
=\left\langle \mathbf{v}^{n}+\sum_{m=1}^{n-1}\lambda_{m}\mathbf{j}%
_{m},A\mathbf{j}_{n}\right\rangle \\
& =\left\langle \mathbf{v}^{n},A\mathbf{j}_{n}\right\rangle
+\sum_{m=1}^{n-1}\lambda_{m}\left\langle \mathbf{j}_{m},A\mathbf{j}_{n}\right\rangle \\
& =\tau_{n}\left\langle \mathbf{v},\eta\mathbf{j}_{n}\right\rangle
,\ \forall\mathbf{v}\mathbf{\in}\hsol ,
\end{eqnarray*}
where we have exploited $\left\langle \mathbf{j}_{m},A\mathbf{j}_{n}\right\rangle =\left\langle A\mathbf{j}%
_{m},\mathbf{j}_{n}\right\rangle =\tau_{m}\left\langle \eta\mathbf{j}%
_{m},\mathbf{j}_{n}\right\rangle =0$ and $\left\langle \mathbf{v}%
^{n},A\mathbf{j}_{n}\right\rangle =\tau_n\left\langle \mathbf{v}%
^{n},\eta\mathbf{j}_{n}\right\rangle =\tau_n\left\langle \mathbf{v},\eta
\mathbf{j}_{n}\right\rangle $. This latter statement gives {\it (v)} and {\it (vi)}.
\end{proof}

The eigenvalues $\{\tau_n(\eta)\}_{n\in\N}$ form a monotonically decreasing sequence converging to zero.
\begin{prop}\label{infinitesima}
The sequence $\left\{  \tau_{n}(\eta)\right\}  _{n\in%
\mathbb{N}
}$ is monotone and decreasing to $0$.
\end{prop}
\begin{proof}
In order to prove that $\tau_{n}\geq\tau_{n+1}$ it is sufficient to observe that the
sets $S_{n}=\left\{  \mathbf{v}\in\hsoln |\left\Vert
\mathbf{v}\right\Vert _{\eta}=1\right\}  $ form a nested family. Indeed,
$S_{n}\supseteq S_{n+1}$ and, thus%
\[
\tau_{n}=\sup_{\mathbf{v}\in S_{n}}\langle  A\mathbf{v}, \mathbf{v}\rangle
\geq\sup_{\mathbf{v}\in S_{n+1}}\langle A  \mathbf{v},   \mathbf{v} \rangle
=\tau_{n+1}.
\]
Since the sequence $\left\{  \tau_{n}\right\}  _{n\in%
\mathbb{N}
}$ is monotonically decreasing and non-negative, its limit exists and it is non-negative:%
\[
\lim_{n\rightarrow+\infty}\tau_{n}=\alpha\geq0.
\]
In order to prove that $\alpha=0$ we observe that sequence $\left\{  \mathbf{j}%
_{n}\right\}  _{n\in%
\mathbb{N}
}$ is bounded in ${L}^{2}\left(  \Omega;\R^3\right)  $. Indeed, $1=\left\Vert
\mathbf{j}_{n}\right\Vert _{\eta}=\left\Vert \sqrt{\eta}\mathbf{j}%
_{n}\right\Vert \geq\sqrt{\eta_{L}}\left\Vert \mathbf{j}_{n}\right\Vert $,
i.e. $\left\Vert \mathbf{j}_{n}\right\Vert \leq1/\sqrt{\eta_{L}}$, where $\eta_L=\textrm{ess}\inf\eta$. Since
$\left\{  \mathbf{j}_{n}\right\}  _{n\in%
\mathbb{N}
}$ is bounded in ${L}^{2}\left(  \Omega;\R^3\right)  $ and $A$ is
compact, there exists a subsequence $\left\{  \mathbf{j}_{n_{k}}\right\}
_{k\in%
\mathbb{N}
}$ such that $A\mathbf{j}_{n_{k}}\rightarrow\mathbf{z\in
}{L}^{2}\left(  \Omega;\R^3\right)  $. Therefore, $\left\{  A\mathbf{j}_{n_{k}}\right\}  _{k\in%
\mathbb{N}
}$ is a Cauchy sequence:
\begin{eqnarray*}
\forall\varepsilon>0\ \exists M_{\varepsilon}\in\mathbb{N}\ \textrm{such that}\ \left\Vert A\mathbf{j}_{n_{k}}-A\mathbf{j}_{n_{l}}\right\Vert <\varepsilon\ \forall k,l>M_{\varepsilon}.
\end{eqnarray*}
As a consequence,%
\[
\left\vert \left\langle A\mathbf{j}_{n_{k}%
}-A\mathbf{j}_{n_{l}},\mathbf{w}\right\rangle \right\vert \leq
\varepsilon\left\Vert \mathbf{w}\right\Vert .
\]
By choosing $\mathbf{w}$ in $\hsol$ we have%
\[
\left\vert \left\langle \tau_{n_{k}}\eta\mathbf{j}_{n_{k}}%
-\tau_{n_{l}}\eta\mathbf{j}_{n_{l}},\mathbf{w}\right\rangle \right\vert \leq
\varepsilon\left\Vert \mathbf{w}\right\Vert ,
\]
thus, for $\mathbf{w}=\tau_{n_{k}}\mathbf{j}_{n_{k}}-\tau_{n_{l}}%
\mathbf{j}_{n_{l}}$ we have%
\[
\left\vert \left\langle \eta\left(  \tau_{n_{k}}\mathbf{j}_{n_{k}}-\tau_{n_{l}%
}\mathbf{j}_{n_{l}}\right) ,\tau_{n_{k}}\mathbf{j}_{n_{k}}-\tau_{n_{l}}%
\mathbf{j}_{n_{l}} \right\rangle \right\vert \leq\varepsilon
\left\Vert \tau_{n_{k}}\mathbf{j}_{n_{k}}-\tau_{n_{l}}\mathbf{j}_{n_{l}%
}\right\Vert
\]
which gives%
\begin{eqnarray*}
\tau_{n_{k}}^{2}+\tau_{n_{l}}^{2}  & \leq\varepsilon\left(  \tau_{n_{k}%
}\left\Vert \mathbf{j}_{n_{k}}\right\Vert +\tau_{n_{l}}\left\Vert
\mathbf{j}_{n_{l}}\right\Vert \right) \\
& \leq\frac{\varepsilon}{\sqrt{\eta_{L}}}\left(  \tau_{n_{k}}+\tau_{n_{l}%
}\right)  \ \forall k,l>M_{\varepsilon}.
\end{eqnarray*}
In the limit for $k\rightarrow+\infty$ we have $\alpha^{2}+\tau_{n_{l}}%
^{2}\leq\varepsilon\left(  \alpha+\tau_{n_{l}}\right)  /\sqrt{\eta_{L}%
}$, \ $\forall l>M_{\varepsilon}$. In the limit for $l\rightarrow+\infty$ we have
$2\alpha^{2}\leq 2\varepsilon\alpha/\sqrt{\eta_{L}}$, i.e. $0\leq\alpha
\leq\varepsilon/\sqrt{\eta_{L}}$. From the arbitrariness of $\varepsilon>0$, it
turns out that $\alpha=0$.
\end{proof}


Sequence $\{{\bf j}_n\}_{n\in\N}$ is important because its elements form a complete basis which can be used to represent any element of $\hsol$ in terms of a Fourier series.

\begin{prop}\label{prop_complete}
The sequence
$\left\{  \mathbf{j}_{n}\right\}  _{n\in%
\mathbb{N}
}$ of the maximizers of (\ref{n-minprob}) forms a complete basis in $\hsol$.
\end{prop}

\begin{proof}
Let us assume that there exists $\mathbf{v\in}\hsol
\backslash \textrm{span}\left\{  \mathbf{j}_{n}\right\}  _{n\in%
\mathbb{N}
}$. Let us decompose $\mathbf{v}$ as $\mathbf{v=v}_{1}+\mathbf{v}_{2}$ with
$\mathbf{v}_{1}\in \textrm{span}\left\{  \mathbf{j}_{n}\right\}  _{n\in%
\mathbb{N}
}$\ and $\mathbf{v}_{2}\neq\mathbf{0}$ such that $\left\langle \mathbf{v}_{2},\eta
\mathbf{j}_{n}\right\rangle =0$, $\forall\ n\in%
\mathbb{N}
$. Thanks to the latter, $\mathbf{v}_{2}\in \hsoln $
$\forall n\in%
\mathbb{N}
$, thus%
\[
\tau_{n}=\sup_{\left\Vert \mathbf{v}\right\Vert _{\eta}=1,\ \mathbf{v}\in
\hsoln  }\langle A \mathbf{v}, \mathbf{v}\rangle
\geq\left\langle A  \frac{\mathbf{v}_{2}}{\left\Vert \mathbf{v}%
_{2}\right\Vert _{\eta}},  \frac{\mathbf{v}_{2}}{\left\Vert \mathbf{v}%
_{2}\right\Vert _{\eta}}\right\rangle  >0
\]
where the last inequality comes from the positive definiteness of
${A}$. In the limits for $n\rightarrow+\infty$ we get $0>0$ which
is a contradiction. Thus, $\mathbf{v}_{2}=\mathbf{0}$, i.e. $\left\{
\mathbf{j}_{n}\right\}  _{n\in%
\mathbb{N}
}$ is complete basis in $\hsol$.
\end{proof}

\section{Monotonicity of eigenvalues}\label{MoE}

In this section we prove the Monotonicity Principle for the time constants. In order to achieve this, we first derive a variational characterization of the time constants in a form slightly different from that in (\ref{n-minprob}). Specifically, we derive a variational characterization (\ref{second_charac}) having two specific features: (i) it involves a finite dimensional space, rather than an infinite dimensional space and (ii) the set of admissible functions does not depend upon the electrical resistivity $\eta$. The first result is obtained in Lemma \ref{second_lem1} whereas the second result is achieved in Lemma \ref{second_lem2}.

\subsection{Variational Characterization}

Let the linear space $U_{n}$ be defined as $U_{n}=$span$\left\{
\mathbf{j}_{1},\mathbf{j}_{2},\ldots,\mathbf{j}_{n}\right\}  $.

\begin{lem}
\label{second_lem1}
The following variational characterization of $\tau_{n}(\eta)$ holds:%
\begin{equation}
\tau_{n}(\eta)=\min_{\mathbf{j}\in U_{n}}\frac{\langle A\mathbf{j}, \mathbf{j}\rangle
}{||  \mathbf{j}||_\eta^2 }.\label{eq07}%
\end{equation}
\end{lem}

\begin{proof}
First, we notice that problem (\ref{eq07}) can be cast in terms of a generalized
eigenvalues problem for matrices. To this end, we express the elements of
$U_{n}$ as%
\[
\mathbf{v=}\sum_{i=1}^{n}x_{i}\mathbf{j}_{i},\ x_{i}\in%
\mathbb{R}
.
\]
Then, we define%
\begin{eqnarray*}
A_{ij}  & =\left\langle A\mathbf{j}_{i},\mathbf{j}%
_{j}\right\rangle \\
B_{ij}  & =\left\langle \eta\mathbf{j}_{i},\mathbf{j}_{j}\right\rangle .
\end{eqnarray*}
Since $U_{n}\sim%
\mathbb{R}
^{n}$ and%
\[
\frac{\langle A \mathbf{v}, \mathbf{v}\rangle}{||
\mathbf{v}||_\eta^2 }=\frac{\mathbf{x}^{T}\mathbf{Ax}}{\mathbf{x}%
^{T}\mathbf{Bx}},
\]
we have 
\[
\min_{\mathbf{v}\in U_{n}}\frac{\langle A  \mathbf{v}, \mathbf{v}\rangle
}{|| \mathbf{v}||_\eta^2}=\min_{\mathbf{x}\in%
\mathbb{R}
^{n}}\frac{\mathbf{x}^{T}\mathbf{Ax}}{\mathbf{x}^{T}\mathbf{Bx}}.
\]
Thanks to Proposition \ref{lemma04}, we have
\begin{eqnarray*}A_{ij}  & =\tau_{i}\delta_{ij}\\B_{ij}  & =\delta_{ij},
\end{eqnarray*}
where $\delta_{ij}$ is the Kronecker delta. From this, it immediately follows
that%
\begin{eqnarray}
\label{finite_charac}
\min_{\mathbf{x}\in%
\mathbb{R}
^{n}}\frac{\mathbf{x}^{T}\mathbf{Ax}}{\mathbf{x}^{T}\mathbf{Bx}}=\tau_{n}.
\end{eqnarray}

\end{proof}
We observe that the variational characterization can also be cast in terms of matrices (see (\ref{finite_charac})).
 
\begin{lem}
\label{second_lem2}The following max-min variational characterization of $\tau
_{n}(\eta)$ holds:%
\begin{eqnarray}
\label{second_charac}
\tau_{n}(\eta)=\max_{\dim\left(  U\right)  =n}\min_{\mathbf{j}\in U}\frac
{\langle A  \mathbf{j}, \mathbf{j}\rangle  }{|| \mathbf{j}||_\eta^2
}.
\end{eqnarray}
\end{lem}
\begin{proof}
For a given $n$ dimensional linear subspace $U$, we have a non-vanishing
$\mathbf{v}\in U\cap\textrm{span}\left\{  \mathbf{j}_{n},\mathbf{j}_{n+1}%
,\ldots\right\}  $. Then%
\[
\min_{\mathbf{j}\in U}\frac{\langle A  \mathbf{j},  \mathbf{j}\rangle
}{|| \mathbf{j}||_\eta^2  }\leq\frac{\langle A
\mathbf{v},
\mathbf{v}\rangle  }{|| \mathbf{v}||_\eta^2 }\leq
\max_{\mathbf v\in\textrm{span}\left\{  \mathbf{j}_{n},\mathbf{j}_{n+1}%
,\ldots\right\}  }\frac{\langle A  \mathbf{j},\mathbf{j}\rangle  }{||\mathbf{j}||_\eta^2  }=\tau_{n}
\]
and, therefore,%
\[
\max_{\dim\left(  U\right)  =n}\min_{\mathbf{j}\in U}\frac{\langle A
\mathbf{j},
\mathbf{j}\rangle  }{|| \mathbf{j}||_\eta^2 }\leq\tau_{n}.
\]
On the other hand, from (\ref{eq07}) we have%
\[
\max_{\dim\left(  U\right)  =n}\min_{\mathbf{j}\in U}\frac{\langle A
\mathbf{j},
\mathbf{j}\rangle  }{|| \mathbf{j}||_\eta^2 }\geq
\min_{\mathbf{j}\in U_{n}}\frac{\langle A \mathbf{j},\mathbf{j}\rangle
}{||  \mathbf{j}||_\eta^2  }=\tau_{n}.
\]
\end{proof}

\subsection{The Monotonicity Principle for the Eigenvalues}
Here we prove the main result of this paper, i.e. the Monotonicity of the time constants with respect to the electrical resistivity $\eta$. A key role is played by the variational characterization (\ref{second_charac}) appearing in Lemma \ref{second_lem2}.

\begin{thm}\label{monotonicity_thm}
Let $\eta_{1},\eta_{2}\in L^{\infty}_+\left(  \Omega\right)$. It holds that%
\[
\eta_{1} \leq\eta_{2}
\ \textrm{a.e. in}\ \Omega\quad \Longrightarrow\quad\tau_{n}\left(  \eta_{1}\right)  \geq\tau
_{n}\left(  \eta_{2}\right)  \ \forall n\in%
\mathbb{N}
,
\]
$\tau_{n}\left(  \eta_{1}\right)$ being  the $n-$th eigenvalue related to
$\eta_{1}$ and $\tau_{n}\left(  \eta_{2}\right)  $ the one related to
$\eta_{2}$.
\end{thm}

\begin{proof}
First, we observe that if $\eta_{1}\leq\eta_{2}\ $a.e. in $\Omega$, then
$||\mathbf{v}||_{\eta_1}^2  =\left\langle \eta_{1}%
\mathbf{v},\mathbf{v}\right\rangle \leq\left\langle \eta_{2}\mathbf{v}%
,\mathbf{v}\right\rangle \leq||  \mathbf{v}||_{\eta_2}  $.
Then, $\langle A
\mathbf{v},
\mathbf{v}\rangle  /||
\mathbf{v}||_{\eta_1}  \langle A
\mathbf{v},
\mathbf{v}\rangle   /|| \mathbf{v}||_{\eta_2}$ and%
\[
\min_{\mathbf{v}\in U}\frac{\langle A
\mathbf{v},
\mathbf{v}\rangle 
}{|| \mathbf{v}||_{\eta_1}  }\geq\min_{\mathbf{v}\in U}%
\frac{\langle A
\mathbf{v},
\mathbf{v}\rangle  }{||
\mathbf{v}||_{\eta_2}  },
\]
where $U$ is a linear space. Eventually, from Lemma \ref{second_lem2}, we have%
\[
\tau_{n}\left(  \eta_{1}\right)  =\max_{\dim\left(  U\right)  =n}%
\min_{\mathbf{v}\in U}\frac{\langle A
\mathbf{v},
\mathbf{v}\rangle 
}{||  \mathbf{v}||_{\eta_1}  }\geq\max_{\dim\left(  U\right)
=n}\min_{\mathbf{v}\in U}\frac{\langle A
\mathbf{v},
\mathbf{v}\rangle 
}{|| \mathbf{v}||_{\eta_2}  }=\tau_{n}\left(  \eta_{2}\right)
,
\]
for any $n\in%
\mathbb{N}
$.
\end{proof}

\section{Interpretation of the results}\label{Dfr}
In this section we discuss the meaning of the results derived in the previous Sections. Specifically, we discuss the relevance of the completeness and orthogonality of the basis $\{{\bf j}_n\}_{n\in\N}$ and the convergences to zeros of the time constants $\tau_n(\eta)$.

\subsection{The Completeness of the basis} Basis $\{{\bf j}_n\}_{n\in\N}$ is complete in $\hsol$ as proved in (\ref{prop_complete}). 
 {Here we prove that any $\mathbf{J\in}L^{2}\left(  0,T;H_{L}\left(
\Omega\right)  \right)  $ can be represented by means of the following Fourier
series:%
\begin{eqnarray}
\label{sum_f}
\mathbf{J}\left(x,t\right)  =\sum_{n=1}^\infty i_n\left(t\right)\ {\bf j}_n (x)\quad\textrm{in}\ \Omega\times [0,T].
\end{eqnarray}
Hereafter we define%
\begin{eqnarray*}
\fl
\left\langle \left\langle \mathbf{u},\mathbf{v}\right\rangle \right\rangle
_{\eta}  & =\int_{0}^{T}\int_{\Omega}\eta\left(  x\right)  \mathbf{u}\left(
x,t\right)  \cdot\mathbf{v}\left(  x,t\right)  \textrm{d}V\textrm{d}t\quad \forall {\bf u}, {\bf v}\in L^2(0,T;\hsol).
\end{eqnarray*}}
 {
Before proceeding with the proof of (\ref{sum_f}), we observe that:
}
 {
\begin{lem}
The factorized function $i_{n}\left(  t\right)  \mathbf{j}_{n}\left(
x\right)  $ is an element of $L^{2}\left(  0,T;H_{L}\left(  \Omega\right)
\right)  $ if\ and only if $i_{n}\in L^{2}\left(  0,T\right)  .$
\end{lem}
}

 {
When $T$ is finite, $L^{2}\left(  0,T\right)  $ admits a countable Fourier
basis $\left\{  f_{k}\right\}  _{k\in%
\mathbb{N}
}$. Therefore $i_{n}\left(  t\right)  =\sum_{k=1}^{+\infty}a_{n,k}f_{k}\left(
t\right)  $ and, consequently, proving (\ref{sum_f}) is equivalent to proving%
\begin{equation*}
\mathbf{J}\left(  x,t\right)  =\sum_{n,k=1}^{+\infty}a_{n,k}f_{k}\left(
t\right)  \mathbf{j}_{n}\left(  x\right) \quad\textrm{in}\ \Omega\times [0,T].
\end{equation*}
The following Theorem holds.
\begin{thm}
\label{ThCB}
For any $0<T<+\infty$, the set $\left\{  f_{k}\left(  t\right)  \mathbf{j}_{n}\left(
x\right)  \right\}  _{k,n\in\mathbb{N}
}$ forms a complete basis of $L^{2}\left(  0,T;H_{L}\left(
\Omega\right)  \right)$.
\end{thm}
\begin{proof}
Let $\mathbf{w}\left(  x,t\right)  \in L^{2}\left(  0,T;H_{L}\left(
\Omega\right)  \right)  \backslash$span$\left\{  f_{k}\mathbf{j}_{n}\right\}
_{k,n\in\mathbb{N}
}$ be a nonvanishing vector field. Let $\mathbf{w}$ be decomposed as
$\mathbf{w=w}_{1}\mathbf{+w}_{2}$ with $\mathbf{w}_{1}\in$span$\left\{
f_{k}\mathbf{j}_{n}\right\}  _{k,n\in%
\mathbb{N}
}$ and $\mathbf{w}_{2}\neq\mathbf{0}$ orthogonal to span$\left\{
f_{k}\mathbf{j}_{n}\right\}  _{k,n\in\mathbb{N}
}$. From $\left\langle \left\langle \mathbf{w}_{2},f_{k}\mathbf{j}%
_{n}\right\rangle \right\rangle _{\eta}=0$, for any $k,n\in\N$, we have%
\begin{equation}
0=\int_{0}^{T}f_{k}\left(  t\right)  \int_{\Omega}\eta\left(  x\right)
\mathbf{w}_{2}\left(  x,t\right)  \cdot\mathbf{j}_{n}\left(  x\right)
\textrm{d}V\textrm{d}t.\label{eq03}%
\end{equation}
Since (\ref{eq03}) is valid for any $k\in\N$ and $\left\{  f_{k}\right\}  _{k\in\mathbb{N}}$ is a complete basis, it follows that $\int_{\Omega}\eta\left(  x\right)
\mathbf{w}_{2}\left(  x,t\right)  \cdot\mathbf{j}_{n}\left(  x\right)  $d$V=0$
a.e. in $]0,T[  $. The latter, being valid for any $n\in\N$, together with the completeness of basis $\left\{\mathbf{j}_{n}\right\}_{n\in\mathbb{N}}$ in $H_{L}\left(  \Omega\right)$, implies that $\mathbf{w}_{2}\left(  x,t\right)  =\mathbf{0}$  a.e. in
$\Omega\times]0,T[$, which contradicts the assumption
$\mathbf{w}_{2}\neq\mathbf{0}$.
\end{proof}
}  {
\begin{rem}
Theorem \ref{ThCB} can be stated as follows:
\begin{equation*}
L^{2}\left(0,T;H_{L}\left(\Omega\right)\right)=\overline{L^2\left(0,T\right)\otimes\hsol}.    
\end{equation*}
\end{rem}
}

 {
It is worth noting that when $\left\{f_{k}\right\}_{k\in \mathbb{N}}$ is an orthonormal basis, then $\left\{  f_{k}\left(  t\right)  \mathbf{j}_{n}\left(
x\right)  \right\}  _{k,n\in\mathbb{N}
}$ forms an orthonormal basis with respect to the weighted inner product
$\left\langle \left\langle \mathbf{\cdot,\cdot}\right\rangle \right\rangle_{\eta}$.}

\subsection{Orthogonality and decoupling}
Function $i_n$ can be found by solving an ordinary differential equation that can be obtained by replacing representation (\ref{sum_f}) in (\ref{weak_form}) and selecting ${\bf w}={\bf j}_n$:
\begin{eqnarray}
\label{ode_f2}
r_n i_n+l_n i_n'=\mathcal{E}_n \quad\forall n\in\N.
\end{eqnarray}
In (\ref{ode_f2}) $r_n=\left\langle \eta {\bf j}_n, {\bf j}_n\right\rangle$, $l_n=\left\langle A {\bf j}_n,{\bf j}_n \right\rangle$ and $\mathcal{E}_n=-\langle\partial_{t}{\bf A}_S, {\bf j}_n\rangle$.

The system of ODEs in (\ref{ode_f2}) is clearly decoupled thanks to the orthogonality properties of $\{{\bf j}_n\}_{n\in\N}$ proved in Proposition \ref{lemma04} (see (v) and (vi)). When using a basis different from $\{{\bf j}_n\}_{n\in\N}$, the ODEs of (\ref{ode_f2}) are no longer decoupled, i.e.
\begin{eqnarray*}
\sum_{k=1}^{+\infty} \left( r_{nk} i_k+l_{nk} i_k'\right)={\mathcal{E}_n} \quad\forall n\in\N,
\end{eqnarray*}
where $r_{nk}=\left\langle \eta {\bf j}_k, {\bf j}_n\right\rangle$ and $l_{nk}=\left\langle A {\bf j}_k,{\bf j}_n \right\rangle$.

Furthermore, when the source current $\mathbf{J}_s$ is vanishing, we have that $\mathbf{A}_s=\mathbf 0$ and, hence, $\mathcal{E}_n=0$ for any $t>0$. In this case, the solution of (\ref{ode_f2}) is given by
\begin{equation}\label{i_n(t)}
i_n(t)=i_n(0)e^{-t/\tau_n(\eta)}
\end{equation}
where we exploited $l_n/r_n=\tau_n(\eta)$, as follows from (\ref{weak_higher}). Equation (\ref{i_n(t)}), together with (\ref{sum_f}), gives (\ref{eq20Su}).

\subsection{Circuit interpretation}
 {Hereafter we assume the unit of $i_n$ to be that of an electrical current (A), in the International System of Units (SI). The unit of $\textbf{j}_n$ is, therefore, the inverse of an area ($\textrm{m}^{-2}$).}

 {This choice is convenient in view of the following circuit interpretation.} Specifically, equation (\ref{ode_f2}) can be interpreted in term of an electrical circuit, where $r_n$ corresponds to a resistor, $l_n$ corresponds to an inductor and $\mathcal{E}_n$ to a voltage generator, as showed in Figure \ref{ode_fig}.

 \begin{figure}[!ht]
	\centering
	\includegraphics[width=0.3\textwidth]{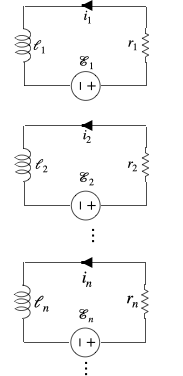}
		\caption{Interpretation of equation (\ref{ode_f2}) in terms of electrical circuit.
		}
	\label{ode_fig}
\end{figure}

When a source electrical current $i_S$ circulates in a coil in a prescribed position of the space, it results that $\textbf{A}_S\left(t,x\right)=
i_S\left(t\right)\textbf{a}_S\left( x\right)$. Physically, ${\bf a}_S={\bf a}_S\left( x\right)$ is the vector potential related to a constant unitary current ($i_S(t)=1$). In this case, which is very common in practical applications, we have that $\mathcal{E}_n=-m_ni_n'$, where $m_n=\langle{\bf a}_S,{\bf j}_n\rangle$ represents a mutual coupling. The electrical circuit of Figure \ref{ode_fig} becomes that sketched in Figure \ref{fig5}.

\begin{figure}[!ht]
	\centering
	\includegraphics[width=0.95\textwidth]{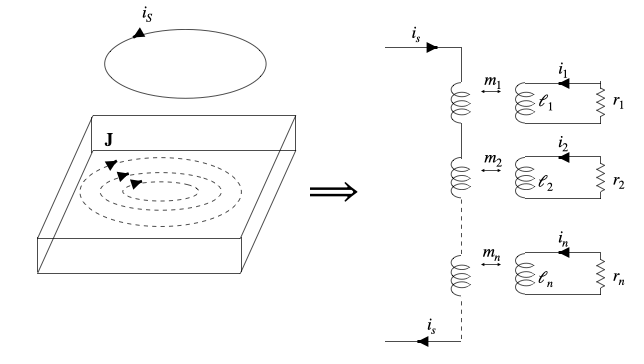}
		\caption{The decoupled systems. $i_c$ is the inducted current produced by the vector potential ${\bf A}_s$ produced by the source.
		}
	\label{fig5}
\end{figure}

\subsection{Ill-posedness of the problem}
From the perspective of the inverse problem, that is the reconstruction of $\eta$ from the knowledge of the time constants, Proposition \ref{infinitesima} represents the \lq\lq signature\rq\rq of the  ill-posedness. 
Indeed, the time constants $\tilde\tau_n$ measured from experimental data are affected by the noise, that is $\tilde\tau_n=\tau_n+\delta_n$, where $\tau_n$ is the noise-free time constant and $\delta_n$ is the noise. If the noise samples $\delta_n$ are in the order of $\Delta>0$, all time constants smaller than $\Delta$ are not reliable and cannot be used by the imaging algorithm.
Since the sequence $\left\{\tau_n\right\}_{n\in\N}$ is monotonic and approaches zero as $n\to\infty$, only a finite number of time constants are larger than the noise level $\Delta$ and can be processed by the imaging algorithm. This means that without any prior information, we can reconstruct only an approximation of the unknown $\eta$ described by a finite number of unknowns parameters.

\subsection{Shape Reconstruction, converse of Monotonicity and Bounds}
 {An important application of MP refers to the inverse obstacle problem, where the unknown is the shape of one or more inclusions in a background medium. Let $\eta_{BG}$ be the electrical resistivity of the background medium and let $\eta_I$ be the electrical resistivity of inclusions. For the sake of simplicity, we assume that both $\eta_{BG}$ and $\eta_I$ are constant, but this assumption can be relaxed. In addition, we assume $\eta_I > \eta_{BG}$; the other case ($\eta_I < \eta_{BG}$) can be treated similarly.
}

 {If $A$ is the region occupied by an inclusion, the related electrical resistivity is $\eta_A\left(x\right)=\eta_I\,\chi_A\left(x\right)+\eta_{BG}\,\chi_{\Omega\setminus A}\left(x\right)$ in $\Omega$. Then, Theorem \ref{monotonicity_thm} implies the following form of MP:
\begin{equation}
    A \subset B \subset \Omega \Longrightarrow \tau_n\left(B\right) \leq \tau_n\left(A\right)\quad \forall n \in \N.
    \label{Inv_Obs_Prob}
\end{equation}
}

 {Proposition (\ref{Inv_Obs_Prob}) can be turned into an imaging algorithm and, under proper conditions, upper and lower bounds are available, even in the presence of noise (see \cite{Tamburrino2016284}, based on \cite{Tamburrino_2002,harrach2015resolution}). That is, if $V \subset \Omega$ is an inclusion, the MP algorithm provides two subsets $V_I$ and $V_U$ such that $V_I \subseteq V \subseteq V_U$.
}

 {
When the converse of (\ref{Inv_Obs_Prob}) or similar is available, as in \cite{harrach2013monotonicity,eberle2020shape,harrach2019monotonicity,griesmaier2018monotonicity,albicker2020monotonicity,daimon2020monotonicity, harrach2019monotonicity-based,harrach2020monotonicity-based,harrach2015combining}, then the imaging method based on MP provides a full characterization of the inclusion $V$. When the converse of MP is not known, as for Magnetic Induction Tomography, then the bounds provide a valuable tool to evaluate, a posteriori, the quality of a reconstruction. Indeed, if $V_I$ is close \lq\lq enough\rq\rq\ to $V_U$, then $V_I$ and $V_U$ constitute a proper reconstruction. Of course, this does not exclude that $V_I=\emptyset$ or $V_U=\Omega$, which would require the converse of MP.
}

\subsection{Further comments}
We point out that the mathematical treatment of Section remains valid in $\R^N$ for $N$ higher than or equal to three, where the kernel of the operator $A$ is equal to the fundamental solution of the Laplace equation.
Finally, we observe that these results can be adapted by minor changes to other problems governed by a diffusion equation \cite{evans2010partial}, such as Optical Diffusive Tomography \cite{jiang2018diffuse} and Thermal Tomography \cite{maldague2012nondestructive}.

\section{Conclusions}\label{C}
In this paper, we have treated an inverse problem consisting of reconstructing the electrical resistivity of a conducting material starting from the free response in the Magneto-Quasi-Stationary limit, a diffusive phenomenon. This inverse problem has been treated in the framework of the Monotonicity Principle.

Specifically, we have proved that the current density can be represented in terms of a modal series, where the dependence upon time and space is factorized (Section \ref{Dfr}). This result has been achieved by finding a proper modal basis in $\hsol$ (Section \ref{TEP}). This key result is based on the analysis of an elliptic eigenvalue problem, following the separation of variables. Then, we find that the free response of the system is characterized by the time constants/eigenvalues $\tau_n(\eta)$ as in (\ref{eq20Su}) (see Section \ref{Dfr}). Consequently, the measured quantities are characterized by the same set of time constants (see (\ref{eq21Su}) and (\ref{eq22Su})).

Moreover, for this continuous setting we have proved the Monotonicity for the operator mapping the electrical resistivity $\eta$ into the countable (ordered) set of the time constants for the source free response. This entails the possibility of also applying the MPM as an imaging method in the continuous setting.

Finally, we highlight that the Monotonicity Principle has been found in many different physical problems governed by diverse PDEs. Despite its rather general nature, each different physical/mathematical context requires the discovery of the proper operator showing the MP with an ad-hoc approach, such as the one presented in this paper.

\section*{Acknowledgements}
This work has been partially supported by the MiUR-Dipartimenti di Eccellenza 2018-2022 grant \lq\lq Sistemi distribuiti intelligenti\rq\rq of Dipartimento di Ingegneria Elettrica e dell'Informazione \lq\lq M. Scarano\rq\rq, by the MiSE-FSC 2014-2020 grant \lq\lq SUMMa: Smart Urban Mobility
Management\rq\rq, by GNAMPA of INdAM and by the CREATE Consortium.

\section*{References}
\bibliographystyle
{iopart-num}
\bibliography{biblioPTZ3}

\end{document}